\documentclass[12pt,reqno]{amsart}
\usepackage[utf8]{inputenc}
\usepackage[margin=1.5 in]{geometry}
\usepackage[lite]{amsrefs}
\usepackage{amsfonts}
\usepackage{mathrsfs}
\usepackage{amscd}

\usepackage{physics}
\usepackage{amsmath}
\usepackage{mathtools}
\usepackage{amssymb}
\usepackage{multicol}
\usepackage{graphicx}
\usepackage{wrapfig}
\usepackage{float}
\usepackage[english]{babel}
\usepackage[utf8]{inputenc}
\usepackage{fancyhdr}
\usepackage{amsthm}
\usepackage[all]{xy}
\usepackage{tikz-cd}
\usepackage{bbold}
\usepackage{hyperref}
\usepackage{xcolor}
\usepackage{tabularx,booktabs,caption,ragged2e}

\newcolumntype{L}{>{\RaggedRight\arraybackslash}X}
\newtheorem{thm}{Theorem}[section]
\newtheorem{cor}[thm]{Corollary}
\newtheorem{lemma}[thm]{Lemma}
\newtheorem{propn}[thm]{Proposition}

\newtheorem*{question}{Question}

\theoremstyle{definition}
\newtheorem{defn}[thm]{Definition}

\theoremstyle{remark}
\newtheorem{remark}[thm]{Remark}

\newcommand{\R}{\mathbb{R}}

\newcommand{\E}{\mathcal{E}}


\title[Finite-rank Lie algebroids for singular foliations]{Finite-Rank Lie Algebroids for Singular Foliations of Prescribed Vanishing Order}

\author{Austin Davis}
\address{Department of Mathematics,
The Pennsylvania State University,
University Park, PA 16802, USA}
\email{\href{mailto:austinhuntdavis@psu.edu}{austinhuntdavis@psu.edu}}

\author{Cheng-Yi Hung}
\address{Department of Mathematics,
National Tsing Hua University,
101, Section 2, Kuang-Fu Road,
Hsinchu 300, Taiwan}
\email{\href{mailto:chengyi@m113.nthu.edu.tw}{chengyi@m113.nthu.edu.tw}}

\author{Sangjun Ko}
\address{Department of Mathematics,
The Pennsylvania State University,
University Park, PA 16802, USA}
\email{\href{mailto:duck@psu.edu}{duck@psu.edu}}

\author{Kaichuan Qi}
\address{Department of Mathematics,
The Pennsylvania State University,
University Park, PA 16802, USA}
\email{\href{mailto:kaichuan@psu.edu}{kaichuan@psu.edu}}

\thanks{Davis, Ko, and Qi were partially supported by the National Science Foundation
(award DMS-2302447) and (award DMS-2201270). Hung gratefully acknowledges the financial support from the National Science and Technology Council (NSTC), Taiwan, through the NSTC Ph.D. Fellowship.}

\begin{document}

\begin{abstract}
We construct a family of finite-rank Lie algebroids realizing, through
their anchor maps, the singular foliations of vector fields vanishing
to a prescribed order at the origin. The construction is obtained from
the sub-adjacent Lie algebroids of a class of left-symmetric
algebroids.
\end{abstract}

\maketitle
\section{Introduction}

It is a general open problem to determine which singular foliations
can be realized as the image sheaf of the anchor of a finite-rank Lie
algebroid. This problem was explicitly raised by
Androulidakis and Zambon \cite{AZ} and has since been highlighted in the
recent account of Laurent-Gengoux, Louis, and Ryvkin \cite{LGLR}. In this
note, we give a construction of finite-rank Lie algebroids
realizing a natural family of singular foliations on finite-dimensional
real vector spaces, consisting of vector fields with a prescribed order
of vanishing at the origin.

Let \(V=\mathbb{R}^n\), with standard coordinates
\(x_1,\ldots,x_n\), and fix an integer \(k\geq 1\).
Denote by \(C^\infty_V\) the sheaf of smooth functions on \(V\),
by \(\mathfrak X_V\) the sheaf of smooth vector fields on \(V\),
and by
$
\mathcal I_0\subset C^\infty_V
$
the {vanishing ideal sheaf of the origin, consisting locally of smooth functions that vanish at $0$}. Consider the singular foliation
\[
\mathcal F_k
=
\mathcal I_0^k\mathfrak X_V.
\]

For a multi-index
\(
\alpha=(\alpha_1,\ldots,\alpha_n)\in\mathbb{N}^n,
\)
write
\[
|\alpha|
=
\alpha_1+\cdots+\alpha_n,
\qquad
x^\alpha
=
x_1^{\alpha_1}\cdots x_n^{\alpha_n}.
\]
By the iterated Hadamard lemma, the ideal sheaf
\(\mathcal I_0^k\) is generated, as a \(C^\infty_V\)-module, by the
monomials
$
x^\alpha$, for
$
|\alpha|=k.
$
Consequently, \(\mathcal F_k\) is globally generated, as a sheaf of
\(C^\infty_V\)-modules, by
\[
x^\alpha\frac{\partial}{\partial x_j},
\qquad
|\alpha|=k,
\qquad
1\leq j\leq n.
\]

\begin{question}
Is this singular foliation the image of the anchor of a
finite-rank Lie algebroid?
\end{question}

For \(k=1\), such a realization is provided by the standard action
Lie algebroid associated with the linear action of
\(\mathfrak{gl}(V)\) on \(V\). The cases \(n\geq 2\) and \(k\geq 2\)
are more subtle. In recent work, Louis \cite{Louis} considered the
same natural presentation of \(\mathcal F_k\) by a finite-rank
anchored vector bundle and constructed an almost Lie algebroid
structure on it, while noting that for \(k\neq 1\) it remained open whether $\mathcal{F}_k$ is the anchor image of a Lie algebroid.

We answer this question positively. More precisely, we construct a
family of finite-rank Lie algebroids whose anchor image is precisely
the singular foliation above. For the special case of \(n=k=2\), this
gives a positive answer to Question 1.8 of \cite{LGLR}. It also gives
a positive answer to the equivalent linear Poisson extension problem
raised by Lerman and Louis \cite{LL}.

{
Consider the trivial vector bundle
\[
A
=
S^kT^*V\otimes TV,
\]
which has rank,
\[
\operatorname{rank}(A)
=
n\binom{n+k-1}{k}.
\]
Thus a section \(a\in\Gamma(A)\) may be viewed as a vector-valued symmetric
covariant \(k\)-tensor field. For vector fields
\(X_1,\ldots,X_k\in\mathfrak X(V)\), we write
\[
a(X_1,\ldots,X_k)\in\mathfrak X(V)
\]
for the vector field defined pointwise by contraction in the $k$ symmetric slots,
\[
a(X_1,\ldots,X_k)_x
=
a_x\bigl(X_1(x),\ldots,X_k(x)\bigr).
\]
}

We first construct an anchor map. Let \(\E\in\mathfrak X(V)\) be the Euler vector field,
given by {\(\E=\sum_j x_j\frac{\partial}{\partial x_j}\)}. There is a natural anchor
$
\rho\colon A\rightarrow TV
$
defined on sections by
\[
\rho(a)
=
a(\underbrace{\E,\ldots,\E}_{k\text{ times}}).
\]

For a multi-index \(\alpha\in\mathbb{N}^n\) with \(|\alpha|=k\), let
\[
dx^{\odot\alpha}
=
dx_1^{\odot\alpha_1}
\odot\cdots\odot
dx_n^{\odot\alpha_n}
\in S^kV^*,
\]
denote the symmetric tensor which satisfies
\(
dx^{\odot\alpha}(\E,\ldots,\E)=x^\alpha.
\)
Thus on generators the anchor map is,
\[
\rho\left(
dx^{\odot\alpha}\otimes
\frac{\partial}{\partial x_j}
\right)
=
x^\alpha\frac{\partial}{\partial x_j}.
\]
{We show in Proposition \ref{propn: anchor-image} that}, for every open subset \(U\subseteq V\), we have 
$
\rho\bigl(\Gamma(U,A)\bigr)
=
\mathcal F_k(U).
$

It remains to construct a Lie algebroid bracket on \(A\) compatible
with this anchor. We make use of the notion of left-symmetric
algebroids \cite{LSBC}, and construct a non-skew-symmetric product
\[
\triangleright\colon
\Gamma(A)\times\Gamma(A)
\longrightarrow\Gamma(A).
\]
We show that it defines a left-symmetric algebroid structure. The
associated sub-adjacent Lie algebroid has bracket
\[
[a,b]_A
=
a\triangleright b-b\triangleright a.
\]


\section{A class of left-symmetric algebroids}

\subsection{Construction of the left-symmetric algebroid}
We first recall basics of left-symmetric algebroids, following Liu, Sheng, Bai, and Chen \cite{LSBC}.
\begin{defn}
Let \(A\to M\) be an anchored vector bundle with anchor
\(\rho\colon A\to TM\). A \emph{left-symmetric algebroid structure}
on \(A\) is an \(\R\)-bilinear product
\[
\triangleright\colon
\Gamma(A)\times\Gamma(A)
\longrightarrow\Gamma(A)
\]
such that, for $f \in C^{\infty}(M)$ and $a,b,c \in \Gamma(A)$,
\begin{align}
(fa)\triangleright b
&=
f(a\triangleright b),
\label{eq:LSA-first-Leibniz}
\\
a\triangleright(fb)
&=
f(a\triangleright b)+\rho(a)[f]\,b,
\label{eq:LSA-second-Leibniz}
\end{align}
and whose associator
\begin{equation}
\operatorname{Ass}(a,b,c)
=
(a\triangleright b)\triangleright c
-
a\triangleright(b\triangleright c)
\label{eq:associator-definition}
\end{equation}
is symmetric in its first two arguments:
\begin{equation}
\operatorname{Ass}(a,b,c)
=
\operatorname{Ass}(b,a,c).
\label{eq:left-symmetric-identity}
\end{equation}
\end{defn}

\begin{propn} \cite{LSBC}
Every left-symmetric algebroid $(A,\rho,\triangleright)$ induces a Lie algebroid $(A,\rho,[\,\cdot,\cdot\,]_A)$, with the Lie bracket
\[
[a,b]_A
=
a\triangleright b-b\triangleright a.
\]
\end{propn}

Now we construct the product $\triangleright$ for the left-symmetric algebroid. We regard \(V\) as an affine manifold equipped with its canonical
flat and torsion-free connection \(D\). We use the same notation
for the induced flat connections on all trivial tensor bundles over
\(V\).  We seek a lift satisfying
\[
\rho(a\triangleright b)
=
D_{\rho(a)}\rho(b).
\]
Since
\[
\rho(b)
=
b(\underbrace{\E,\ldots,\E}_{k\text{ times}}),
\]
{by applying \(D_{\rho(a)}\) to the vector field
\(b(\E,\ldots,\E)\), the covariant derivative acts both on the
tensor field \(b\) and on each of its vector-field arguments. Thus,
by the Leibniz rule for tensor contractions,}

\begin{align*}
D_{\rho(a)}\rho(b)
&=
(D_{\rho(a)}b)(\E,\ldots,\E)
\\
&+
\sum_{i=1}^k
b(\E,\ldots,
\underbrace{D_{\rho(a)}\E}_{i\text{-th entry}},
\ldots,\E).
\end{align*}


{For every \(a\in\Gamma(A)\), define an endomorphism field}
$
\Lambda_a
\in
\Gamma(\operatorname{End}(TV))
$
by
\begin{equation}
\Lambda_a(X)
=
a(\underbrace{\E,\ldots,\E}_{k-1\text{ times}},X), \quad \forall X \in \mathfrak{X}(V).
\label{eq:Lambda}
\end{equation}
Then
\begin{equation}\label{eq:Lambda-anchor}
\Lambda_a\E
=
\rho(a)
=
D_{\rho(a)}\E.    
\end{equation} This suggests defining 

\begin{align}(a\triangleright b)(X_1,\ldots,X_k)&=(D_{\rho(a)}b)(X_1,\ldots,X_k) \nonumber \\ &+
\sum_{i=1}^k
b(X_1,\ldots,X_{i-1},
\Lambda_a X_i,
X_{i+1},\ldots,X_k) \label{eq:left-symmetric-product-expanded}  \end{align}
for $X_i\in \mathfrak{X}(V)$. By construction,
\begin{equation}
\rho(a\triangleright b)
=
D_{\rho(a)}\rho(b).
\label{eq:rho-product-identity}
\end{equation}

For an endomorphism field
$
\Phi\in \Gamma(\operatorname{End}(TV))
$,
define
$
\mathcal P_\Phi\colon\Gamma(A)\rightarrow\Gamma(A)
$
by
\begin{equation}
(\mathcal P_\Phi b)(X_1,\ldots,X_k)
=
\sum_{i=1}^k
b(X_1,\ldots,X_{i-1},
\Phi X_i,
X_{i+1},\ldots,X_k).
\label{eq:P-Phi}
\end{equation}
With this notation,
\[
a\triangleright b
=
D_{\rho(a)}b+\mathcal P_{\Lambda_a}b.
\]
The main result is
\begin{thm}\label{thm:Main}
The product
\[
a\triangleright b
=
D_{\rho(a)}b+\mathcal P_{\Lambda_a}b
\]
defines a left-symmetric algebroid structure on the anchored vector
bundle \(A\to V\).
\end{thm}

\subsection{A Leibniz-rule identity, and commutators of the tensorial operators}

We prepare the proof of Theorem \ref{thm:Main} with the following 2 lemmas.

\begin{lemma}
For all \(a,b\in\Gamma(A)\), we have
\begin{equation}
\Lambda_{a\triangleright b}
=
D_{\rho(a)}\Lambda_b
+
\Lambda_b\circ\Lambda_a.
\label{eq:Lambda-product-identity}
\end{equation}
\end{lemma}

\begin{proof}
{Let \(X\in\mathfrak X(V)\).} Then

\begin{align*}
\Lambda_{a\triangleright b}(X)
&=
(a\triangleright b)
(\underbrace{\E,\ldots,\E}_{k-1\text{ times}},X)
\\
&=
(D_{\rho(a)}b)
(\underbrace{\E,\ldots,\E}_{k-1\text{ times}},X)
\\
&+
\sum_{i=1}^{k-1}
b(\E,\ldots,
\underbrace{\Lambda_a(\E)}_{i\text{-th entry}},
\ldots,\E,X)
\\
&+
b(\underbrace{\E,\ldots,\E}_{k-1\text{ times}},
\Lambda_a(X)).
\end{align*}

The fact that \(b\) is symmetric
in its \(k\) vector arguments implies all \(k-1\) terms in the sum are equal.
Thus, together with \eqref{eq:Lambda-anchor}, we have 
\begin{equation}
\begin{aligned}
\Lambda_{a\triangleright b}(X)
={}&
(D_{\rho(a)}b)
(\underbrace{\E,\ldots,\E}_{k-1\text{ times}},X)
\\
&+
(k-1)b\bigl(
\rho(a),
\underbrace{\E,\ldots,\E}_{k-2\text{ times}},
X
\bigr)
\\
&+
\Lambda_b(\Lambda_a(X)).
\end{aligned}
\label{eq:Lambda-product-first}
\end{equation}
When \(k=1\), the middle term is absent.

{On the other hand, using the induced connection on
\(\operatorname{End}(TV)\), we have}
\[
\begin{aligned}
\bigl(D_{\rho(a)}\Lambda_b\bigr)(X)
={}&
D_{\rho(a)}\bigl(\Lambda_b(X)\bigr)
-
\Lambda_b\bigl(D_{\rho(a)}X\bigr)
\\
={}&
D_{\rho(a)}
\bigl(
b(\underbrace{\E,\ldots,\E}_{k-1\text{ times}},X)
\bigr)
-
\Lambda_b\bigl(D_{\rho(a)}X\bigr).
\end{aligned}
\]
Applying the Leibniz rule and using
\eqref{eq:Lambda-anchor}, this becomes
\[
\begin{aligned}
\bigl(D_{\rho(a)}\Lambda_b\bigr)(X)
={}&
(D_{\rho(a)}b)
(\underbrace{\E,\ldots,\E}_{k-1\text{ times}},X)
\\
&+
(k-1)b\bigl(
\rho(a),
\underbrace{\E,\ldots,\E}_{k-2\text{ times}},
X
\bigr)
\\
&+
b(\underbrace{\E,\ldots,\E}_{k-1\text{ times}},
D_{\rho(a)}X)
-
\Lambda_b\bigl(D_{\rho(a)}X\bigr).
\end{aligned}
\]
The last two terms cancel by the definition of \(\Lambda_b\). Hence
\[
\begin{aligned}
\bigl(D_{\rho(a)}\Lambda_b\bigr)(X)
={}&
(D_{\rho(a)}b)
(\underbrace{\E,\ldots,\E}_{k-1\text{ times}},X)
\\
&+
(k-1)b\bigl(
\rho(a),
\underbrace{\E,\ldots,\E}_{k-2\text{ times}},
X
\bigr).
\end{aligned}
\]
Again, the last term is absent when \(k=1\).
Comparing this with \eqref{eq:Lambda-product-first}, we obtain
\[
\Lambda_{a\triangleright b}(X)
=
\bigl(D_{\rho(a)}\Lambda_b\bigr)(X)
+
(\Lambda_b\circ\Lambda_a)(X).
\]

\end{proof}

For operators on \(\Gamma(A)\), we use the convention
\([\mathcal D_1,\mathcal D_2]=
\mathcal D_1\circ\mathcal D_2-\mathcal D_2\circ\mathcal D_1\).
For endomorphism fields, composition is written in the usual order,
\((\Phi\circ\Psi)(u)=\Phi(\Psi(u))\).

\begin{lemma}
For all \(a,b\in\Gamma(A)\),
\begin{align}
[D_{\rho(a)},D_{\rho(b)}]
&=
D_{[\rho(a),\rho(b)]},
\label{eq:DD-commutator}
\\
[D_{\rho(a)},\mathcal P_{\Lambda_b}]
&=
\mathcal P_{D_{\rho(a)}\Lambda_b},
\label{eq:DP-commutator}
\\
[\mathcal P_{\Lambda_a},\mathcal P_{\Lambda_b}]
&=
\mathcal P_{
\Lambda_b\circ\Lambda_a
-
\Lambda_a\circ\Lambda_b
}.
\label{eq:PP-commutator}
\end{align}
\end{lemma}

\begin{proof}
Identity \eqref{eq:DD-commutator} follows from the flatness of \(D\).

Let \(c\in\Gamma(A)\), and let
\(X_1,\ldots,X_k\) be constant vector fields on $V$. By definition,
\[
(\mathcal P_{\Lambda_b}c)(X_1,\ldots,X_k)
=
\sum_{i=1}^k
c(X_1,\ldots,X_{i-1},
\Lambda_bX_i,
X_{i+1},\ldots,X_k).
\]
Differentiating along \(\rho(a)\), one obtains
\[
\begin{aligned}
&\left(D_{\rho(a)}
\bigl(\mathcal P_{\Lambda_b}c\bigr)\right)
(X_1,\ldots,X_k)
\\
={}&
\sum_{i=1}^k
(D_{\rho(a)}c)
(X_1,\ldots,X_{i-1},
\Lambda_bX_i,
X_{i+1},\ldots,X_k)
\\
&+
\sum_{i=1}^k
c(X_1,\ldots,X_{i-1},
(D_{\rho(a)}\Lambda_b)X_i,
X_{i+1},\ldots,X_k).
\end{aligned}
\]
Meanwhile,
\[
\begin{aligned}
&\left(\mathcal P_{\Lambda_b}
\bigl(D_{\rho(a)}c\bigr)\right)
(X_1,\ldots,X_k)
\\
={}&
\sum_{i=1}^k
(D_{\rho(a)}c)
(X_1,\ldots,X_{i-1},
\Lambda_bX_i,
X_{i+1},\ldots,X_k).
\end{aligned}
\]
Subtracting gives
\[
\begin{aligned}
&\big([D_{\rho(a)},\mathcal P_{\Lambda_b}]c\big)
(X_1,\ldots,X_k)
\\
={}&
\sum_{i=1}^k
c(X_1,\ldots,X_{i-1},
(D_{\rho(a)}\Lambda_b)X_i,
X_{i+1},\ldots,X_k)
\\
={}&
\big(\mathcal P_{D_{\rho(a)}\Lambda_b}c\big)
(X_1,\ldots,X_k),
\end{aligned}
\]
which proves \eqref{eq:DP-commutator}.

For the last identity, expand
\(\mathcal P_{\Lambda_a}\mathcal P_{\Lambda_b}c\) and
\(\mathcal P_{\Lambda_b}\mathcal P_{\Lambda_a}c\).
The terms in which \(\Lambda_a\) and \(\Lambda_b\) act in different
tensor slots occur identically in the two expansions and cancel
after subtraction. The remaining terms, in which the two endomorphisms
act in the same tensor slot, give
\(\mathcal P_{\Lambda_b\circ\Lambda_a-\Lambda_a\circ\Lambda_b}c\).
This proves \eqref{eq:PP-commutator}.
\end{proof}

\subsection{Proof of Theorem \ref{thm:Main}}
\begin{proof}
We first verify the two Leibniz identities.

Let \(f\in C^\infty(V)\). Since
\[
\rho(fa)=f\rho(a),
\qquad
\Lambda_{fa}=f\Lambda_a,
\]
one has
\[
\begin{aligned}
(fa)\triangleright b
&=
D_{\rho(fa)}b+\mathcal P_{\Lambda_{fa}}b
\\
&=
D_{f\rho(a)}b+\mathcal P_{f\Lambda_a}b
\\
&=
fD_{\rho(a)}b+f\mathcal P_{\Lambda_a}b
\\
&=
f(a\triangleright b).
\end{aligned}
\]
Also,
\[
\begin{aligned}
a\triangleright(fb)
&=
D_{\rho(a)}(fb)+\mathcal P_{\Lambda_a}(fb)
\\
&=
\rho(a)[f]\,b
+
fD_{\rho(a)}b
+
f\mathcal P_{\Lambda_a}b
\\
&=
f(a\triangleright b)+\rho(a)[f]\,b.
\end{aligned}
\]
Thus \eqref{eq:LSA-first-Leibniz} and
\eqref{eq:LSA-second-Leibniz} hold.

It remains to verify the left-symmetric identity. Define the commutator
\begin{equation}
[a,b]_A
=
a\triangleright b-b\triangleright a.
\label{eq:subadjacent-bracket}
\end{equation}
We also write
$
\nabla_ab:=a\triangleright b.
\label{eq:nabla-product}$
By definition,
\[
\begin{aligned}
\operatorname{Ass}(a,b,c)
&=
(a\triangleright b)\triangleright c
-
a\triangleright(b\triangleright c)
\\
&=
\nabla_{\nabla_a b}c-\nabla_a\nabla_b c.
\end{aligned}
\]
Therefore
\[
\begin{aligned}
&\operatorname{Ass}(a,b,c)
-
\operatorname{Ass}(b,a,c)
\\
={}&
\nabla_{\nabla_a b}c
-
\nabla_a\nabla_b c
-
\nabla_{\nabla_b a}c
+
\nabla_b\nabla_a c
\\
={}&
\nabla_{\nabla_a b-\nabla_b a}c
-
[\nabla_a,\nabla_b]c
\\
={}&
\nabla_{[a,b]_A}c
-
[\nabla_a,\nabla_b]c.
\end{aligned}
\]
{Therefore showing the symmetry of the associator is equivalent to the flatness identity} \begin{equation}
[\nabla_a,\nabla_b]
=
\nabla_{[a,b]_A}.
\label{eq:flatness-identity}
\end{equation}



Recall that, as an operator on \(\Gamma(A)\),
\(\nabla_a=D_{\rho(a)}+\mathcal P_{\Lambda_a}\). Therefore
\[
\begin{aligned}
[\nabla_a,\nabla_b]
={}&
[D_{\rho(a)},D_{\rho(b)}]
+[D_{\rho(a)},\mathcal P_{\Lambda_b}]
\\
&-
[D_{\rho(b)},\mathcal P_{\Lambda_a}]
+[\mathcal P_{\Lambda_a},\mathcal P_{\Lambda_b}].
\end{aligned}
\]
Using \eqref{eq:DD-commutator}--\eqref{eq:PP-commutator} gives
\begin{align}
[\nabla_a,\nabla_b]
={}&
D_{[\rho(a),\rho(b)]}
\nonumber\\
&+
\mathcal P_{
D_{\rho(a)}\Lambda_b
-
D_{\rho(b)}\Lambda_a
+
\Lambda_b\circ\Lambda_a
-
\Lambda_a\circ\Lambda_b
}.
\label{eq:operator-commutator-expanded}
\end{align}
Using \eqref{eq:Lambda-product-identity},
\begin{align}
D_{\rho(a)}\Lambda_b
-
&~
D_{\rho(b)}\Lambda_a
+
\Lambda_b\circ\Lambda_a
-
\Lambda_a\circ\Lambda_b \nonumber \\
=&~\Lambda_{a\triangleright b}-\Lambda_{b\triangleright a} \nonumber \\
=&~\Lambda_{[a,b]_A} \label{eq:Lambda-bracket-identity},
\end{align}
and since $D$ is torsion free,
\begin{align}
    [\rho(a),\rho(b)]=&~D_{\rho(a)}\rho(b)-D_{\rho(b)}\rho(a)\nonumber\\
    =&~\rho(a\triangleright b)-\rho(b\triangleright a)\nonumber\\
    =&~\rho([a,b]_A) . \label{eq:anchor-identity} \end{align}
Substituting \eqref{eq:Lambda-bracket-identity} and \eqref{eq:anchor-identity} back into
\eqref{eq:operator-commutator-expanded}, we obtain
\[
[\nabla_a,\nabla_b]
=
D_{\rho([a,b]_A)}+\mathcal P_{\Lambda_{[a,b]_A}}
=
\nabla_{[a,b]_A}.
\]
This proves \eqref{eq:flatness-identity}. Therefore \((A,\triangleright,\rho)\) is a left-symmetric algebroid, completing the proof of Theorem~\ref{thm:Main}.
\end{proof}

\subsection{The anchor image}

We now return to the singular foliation
\[
\mathcal F_k
=
\mathcal I_0^k\mathfrak X_V.
\]

\begin{propn}\label{propn: anchor-image}
The image of the anchor agrees with \(\mathcal F_k\). More precisely, for
every open subset \(U\subseteq V\),
\begin{equation}
\rho\bigl(\Gamma(U,A)\bigr)
=
\mathcal F_k(U),
\label{eq:sheaf-anchor-image}
\end{equation}
and
\begin{equation}
\rho\bigl(\Gamma_c(A)\bigr)
=
\mathcal F_k(V)\cap\mathfrak X_c(V).
\label{eq:compact-anchor-image}
\end{equation}
\end{propn}

\begin{proof}
The constant sections
\[
dx^{\odot\alpha}\otimes \frac{\partial}{\partial x_j},
\qquad
|\alpha|=k,\quad 1\leq j\leq n,
\]
form a global frame of \(A\), and their anchor images are
\[
x^\alpha\frac{\partial}{\partial x_j}.
\]
These are precisely the global generators of
\(\mathcal I_0^k\mathfrak X_V\) as a sheaf of
\(C^\infty_V\)-modules.

We now verify that this gives equality on sections over every
open subset \(U\subseteq V\). The inclusion
$
\rho\bigl(\Gamma(U,A)\bigr)\subseteq \mathcal F_k(U)
$
is immediate. Conversely, let \(X\in\mathcal F_k(U)\). Since the vector
fields
\[
x^\alpha\frac{\partial}{\partial x_j},
\qquad
|\alpha|=k,\quad 1\leq j\leq n,
\]
generate \(\mathcal F_k\) as a sheaf, there is an open cover
\(\{U_\ell\}\) of \(U\) and sections
\(a_\ell\in\Gamma(U_\ell,A)\) such that
\[
\rho(a_\ell)=X|_{U_\ell}.
\]
Choose a locally finite partition of unity
\(\{\chi_\ell\}\) subordinate to \(\{U_\ell\}\). Then
\[
a=\sum_\ell \chi_\ell a_\ell\in\Gamma(U,A)
\]
is well defined and, since the anchor is \(C^\infty\)-linear,
\[
\rho(a)
=
\sum_\ell \chi_\ell\rho(a_\ell)
=
\sum_\ell \chi_\ell X
=
X.
\]
Hence, for every open subset \(U\subseteq V\),
\[
\rho\bigl(\Gamma(U,A)\bigr)=\mathcal F_k(U).
\]

The compactly supported statement follows by the usual cutoff argument.
Indeed, if
$
X=\rho(a)\in\mathfrak X_c(V)
$
for some \(a\in\Gamma(A)\), choose
\(\chi\in C^\infty_c(V)\) equal to \(1\) on a neighborhood of
\(\operatorname{supp}(X)\). Then
\[
\chi a\in\Gamma_c(A),
\qquad
\rho(\chi a)=\chi X=X.
\]
Thus
\[
\rho\bigl(\Gamma_c(A)\bigr)
=
\rho\bigl(\Gamma(A)\bigr)\cap\mathfrak X_c(V)
=
\mathcal F_k(V)\cap\mathfrak X_c(V).
\]
\end{proof}

\begin{cor}
For every \(n\geq1\) and \(k\geq1\), the singular foliation
\[
\mathcal F_k
=
\mathcal I_0^k\mathfrak X_{\R^n}
\]
is the anchor image of the finite-rank Lie algebroid
\[
A
=
\R^n\times
\bigl(S^k(\R^n)^*\otimes\R^n\bigr)
\longrightarrow\R^n,
\]
whose rank is
\[
n\binom{n+k-1}{k}.
\]
\end{cor}

\begin{remark}
The problem considered in this note is also related to recent work
of Lavau \cite{L} and of Lerman--Louis \cite{LL}. We hope that the construction presented
here provides new examples, and offers a complementary perspective, for their theories.
\end{remark}

\section{Example: The quadratic case on \texorpdfstring{$\mathbb{R}^2$}{R2}}
\label{sec:quadratic-case}

For \(n=k=2\), write \(x,y\) for the coordinates on \(\mathbb{R}^2\).
The rank-six bundle \(A\) has the global frame
\[
\begin{aligned}
a_2&=dx^{\odot 2}\otimes\partial_x,&
a_1&=(dx\odot dy)\otimes\partial_x,&
a_0&=dy^{\odot 2}\otimes\partial_x,\\
b_0&=dx^{\odot 2}\otimes\partial_y,&
b_1&=(dx\odot dy)\otimes\partial_y,&
b_2&=dy^{\odot 2}\otimes\partial_y.
\end{aligned}
\]
Swapping $x$'s for $y$'s one has $a_j\leftrightarrow b_j$. The anchor satisfies
\[
\begin{aligned}
\rho(a_2)&=x^2\partial_x,&
\rho(a_1)&=xy\partial_x,&
\rho(a_0)&=y^2\partial_x,\\
\rho(b_0)&=x^2\partial_y,&
\rho(b_1)&=xy\partial_y,&
\rho(b_2)&=y^2\partial_y.
\end{aligned}
\]
These vector fields generate \(\mathcal F_2\) as a sheaf.

Since our generators are constant sections, the derivative term in
\eqref{eq:left-symmetric-product-expanded} vanishes. For example,
\[
a_0\triangleright b_0=2yb_1,
\qquad
b_0\triangleright a_0=2xa_1,
\]
so \([a_0,b_0]_A=2yb_1-2xa_1\). The full list of brackets is
\[
\begin{aligned}
	[a_0,a_1]_A&=ya_0,&
			[b_0,b_1]_A&=xb_0,\\
	[a_0,a_2]_A&=2ya_1,&
			[b_0,b_2]_A&=2xb_1,\\
	[a_1,a_2]_A&=ya_2,&
			[b_1,b_2]_A&=xb_2,\\
	~\\
	[b_1,a_0]_A&=xa_0+ya_1-yb_2,&
			[a_1,b_0]_A&=yb_0+xb_1-xa_2,\\
	[b_2,a_0]_A&=2ya_0,&
			[a_2,b_0]_A&=2xb_0,\\
	[b_2,a_1]_A&=ya_1,&
			[a_2,b_1]_A&=xb_1,\\
	~\\
	[a_0,b_0]_A&=2yb_1-2xa_1,&&\\
	[a_1,b_1]_A&=\tfrac12(yb_1+xb_2-ya_2-xa_1),&&\\
	[a_2,b_2]_A&=0.&&
\end{aligned}
\]
The remaining brackets follow from skew symmetry. Together with the anchor and the Leibniz
rule, these relations determine the bracket on arbitrary sections.

As an illustration of the Jacobi identity, consider \(a_2,b_0,\) and \(b_1 \).
Using the Leibniz rule, \(\rho(a_2)(x)=x^2\), and
\(\rho(b_0)(x)=\rho(b_1)(x)=0\), we obtain
\[
\begin{aligned}
[a_2,[b_0,b_1]_A]_A
&=[a_2,xb_0]_A=2x^2b_0+x^2b_0=3x^2b_0,\\
[b_0,[b_1,a_2]_A]_A
&=[b_0,-xb_1]_A=-x^2b_0,\\
[b_1,[a_2,b_0]_A]_A
&=[b_1,2xb_0]_A=-2x^2b_0.
\end{aligned}
\]
The three terms add to zero.

\vspace{20pt}  

\noindent \textbf{Acknowledgments.} We thank Leonid Ryvkin for bringing this problem to our attention
during GAP 2026 at the Research Institute for Mathematical Sciences
(RIMS), Kyoto University. We are grateful to Camille Laurent-Gengoux, Hsuan-Yi Liao, Mathieu Sti\'{e}non, Ping Xu, and Marco Zambon for fruitful discussions and useful comments. We would like to thank the anonymous referees for valuable suggestions.





\end{document}